\documentclass[arxiv,reqno,twoside,a4paper,12pt]{amsart}

\usepackage[utf8]{inputenc}
\usepackage{amsmath}
\usepackage{amsfonts}
\usepackage{amssymb}
\usepackage{amsthm}
\usepackage{float}
\usepackage{mathtools}
\usepackage{csquotes}
\usepackage[numbers]{natbib}
\usepackage{bbold}

\usepackage[colorlinks=true,linkcolor=blue,citecolor=blue]{hyperref}

\usepackage[scaled]{helvet} 
\usepackage{courier} 
\usepackage[mathbf]{euler}
\usepackage[dvipsnames]{xcolor}

\usepackage[driver=pdftex,margin=3cm,heightrounded=true,centering]{geometry}

\usepackage{tikz-cd}
\usetikzlibrary{arrows,matrix,shapes,decorations.text, mindmap,shapes.misc}
\usepackage{metalogo}

\renewcommand{\S}{\mathbb{S}}

\newcommand{\R}{\mathbb{R}}
\newcommand{\norm}[1]{\left\lVert#1\right\rVert}
\newcommand{\ip}[2]{\left\langle #1,#2\right\rangle}

\newcommand{\Vol}[0]{\mathrm{Vol}}

\newcommand{\tF}{\tilde{\mathcal{F}}}
\usepackage[numbers]{natbib}

\newcommand{\sgn}{\mathrm{sgn}}

\tikzset{cross/.style={cross out, draw=black, minimum size=2*(#1-\pgflinewidth), inner sep=0pt, outer sep=0pt},
cross/.default={1pt}}

\theoremstyle{plain}
\newtheorem{Thm}{Theorem}[section]

\newtheorem{Lem}[Thm]{Lemma}
\newtheorem{Cor}[Thm]{Corollary}
\newtheorem{Prop}[Thm]{Proposition}
\newtheorem{Def}[Thm]{Definition}
\newtheorem{Rem}[Thm]{Remark}

\theoremstyle{plain}
\newtheorem{thm}{theorem}
\newtheorem{Assump}[thm]{Assumption}

\numberwithin{equation}{section}

\title{}
\date{}
\author{}

\begin{document}

\title[Generalized Yamabe Flows]
{Generalized Yamabe Flows}

\author{J\o rgen Olsen Lye}
\address{Institut für Differentialgeometrie,
Leibniz Universit\"at Hannover,
Germany}
\email{joergen.lye@math.uni-hannover.de}

\author{Boris Vertman}
\address{Mathematisches Institut,
Universit\"at Oldenburg,
Germany}
\email{boris.vertman@uni-oldenburg.de}

\author{Mannaim Gennaro Vitti}
\address{Departamento de Matemática, Universidade Federal de São Carlos (UFSCar),  Brazil}
\email{manavitti@yahoo.com.br}

\date{\today}

\begin{abstract}
{In this work we introduce a family of conformal flows generalizing the classical Yamabe flow.  We prove that for a large class of such flows long-time existence holds, and the arguments are in fact simpler than in the classical case. Moreover, we establish convergence for the case of negative scalar curvature and expect a similar statement for the positive and the flat cases as well.}
\end{abstract}

\maketitle
\tableofcontents

\section{Introduction and statement of the main results}\label{intro-section}

\subsection{Weak and strong Yamabe conjectures}
Let $M$ be a compact manifold (in this paper always without boundary unless otherwise stated) of dimension $n\geq 3$.
Yamabe \cite{Yamabe} suggested the following min-max procedure for producing Einstein metrics on $M$. 
The procedure consists of two steps. For any Riemannian metric $g_0$ on $M$,
find a minimum $g_{\min}$ of the \emph{Einstein-Hilbert functional} 
\begin{equation}
\mathcal{E}(g)\coloneqq  \Vol_g^{\frac{2-n}{n}} \int_M S\, d\Vol_g 
\label{eq:EH}
\end{equation}
in the conformal class of $[g_0]$ while keeping the total volume fixed. 
Here, $S = S(g)$ denotes the scalar curvature of $g$, $\Vol_g$ is the total volume of $M$ 
with respect to $g$. One sets
\[
\mathcal{E}(g_{\min})=\inf_{g\in [g_0]} \mathcal{E}(g) \eqqcolon Y(M,[g_0]),
\]
where $Y(M,[g_0])$ is known as the \emph{Yamabe constant}.
As the second step, find a maximum $g$ of the Einstein-Hilbert functional amongst all $g_{\min}$.
 The resulting metric, if it exists, is an Einstein metric. Of course, the space of Riemannian metrics is not compact, so the above minimum and maximum may fail to exist. \medskip

The first step, existence of a minimiser of the Einstein-Hilbert action within a conformal class, became known as the \textbf{(strong) Yamabe conjecture}.
A necessary condition for minimising is of course to be a critical point. Critical points of the functional \eqref{eq:EH} within a conformal class are precisely constant scalar curvature metrics. Hence the 
\textbf{(weak)}\footnote{We have referred to the weak Yamabe conjecture as the Yamabe conjecture in previous work. This is also what one finds in \cite{Brendle}. In a discussion of the first author with Kazuo Akutagawa, the latter informed us that in Japan, one typically means the strong Yamabe conjecture when talking about the Yamabe problem. Hence the distinction.} \textbf{Yamabe conjecture}
was born: given a compact Riemannian manifold $(M,g_0)$, can one conformally change the metric $g_0$ to make the scalar curvature constant? \medskip

The first proof of the (strong) conjecture was initiated by Yamabe \cite{Yamabe} and continued by Trudinger \cite{Trudinger},
Aubin \cite{Aubin}, and Schoen \cite{Schoen}. The proof is based on the calculus of variations and elliptic partial differential 
equations. 

\subsection{Normalized Yamabe flow}
An alternative tool for proving the conjecture is due to Hamilton \cite{Hamilton}: the normalized Yamabe flow of a Riemannian manifold $(M,g_0)$. This is a family $g\equiv g(t), t\in [0,T]$ of Riemannian metrics on $M$ such 
that the following evolution equation holds
 \begin{equation}\label{eq:YF}
 \begin{cases}\partial_t g=-(S-\sigma)g, \quad \sigma \coloneqq \Vol_g^{-1} \int_M S\, d\Vol_g, \\
  g(0)=g_0.\end{cases}
 \end{equation}
As above, $S=S(g)$ is the scalar curvature of $g$, $\Vol_g$ the total volume of $M$ with respect to $g$ 
and $\sigma$ is the average scalar curvature of $g$.
The normalization by $\sigma$ ensures that the total volume does not change along the flow.
\medskip

The Yamabe flow \eqref{eq:YF} is the negative gradient flow of \eqref{eq:EH} within a conformal class, subject to the total volume being fixed. Since Yamabe initially wanted to find a minimum of \eqref{eq:EH}, this is a natural approach to try.
Hamilton \cite{Hamilton} showed the long time existence of \eqref{eq:YF}. 
Stationary points of \eqref{eq:YF} are clearly constant scalar curvature metrics. The hope of Hamilton was to show that \eqref{eq:YF} always converges as $t\to \infty$ to a constant scalar curvature metric. This would then prove the (weak) Yamabe conjecture by parabolic methods. This may of course fail to be a global minimum of $\eqref{eq:EH}$, so this would not settle the (strong) Yamabe conjecture without further arguments. 

\subsection{Literature in compact and non-compact settings}
Establishing convergence of the normalized Yamabe flow is intricate already in the setting of smooth, compact manifolds. Ye \cite{Ye} was able to show long time existence and convergence if the initial scalar curvature is negative or the initial metric admits a scalar flat metric in its conformal class. They were also able to show convergence when the initial scalar curvature is positive and the manifold is locally conformally flat.
The case of a non-conformally flat $g_0$ with positive scalar curvature is delicate and has been studied first by Schwetlick and Struwe
\cite{SS} for large energies and later by Brendle \cite{Brendle, BrendleYF} for arbitrary energies. More specifically,
\cite[Section 5]{SS} as well as \cite[p. 270]{Brendle},  \cite[p. 544]{BrendleYF} invoke the positive mass theorem, which is 
where the dimensional restriction in \cite{SS},\cite{Brendle} and the spin assumption in \cite[Theorem 4]{BrendleYF} come from. As to date, we do \textit{not} know that the Yamabe flow converges for an arbitrary compact Riemannian manifold with positive scalar curvature without imposing further restrictions. 
\medskip

In the non-compact setting, our understanding is limited. On complete manifolds, long-time existence 
has been discussed in various settings by Ma \cite{Ma}, Ma and An \cite{MaAn}, Schulz \cite{Schulz}.
On incomplete surfaces, where the Ricci and Yamabe flows coincide, see the work by Isenberg, Mazzeo and Sesum 
\cite{IMS}, Yin \cite{Yin}, as well as Giesen and Topping \cite{Topping, Topping2}
who constructed a flow that becomes instantaneously complete. Analysis of the Yamabe flow on incomplete spaces 
with cone-edge (wedge) singularities has been initiated by Bahuaud and the second author in 
\cite{ShortTime, LongTime}, where existence and convergence of the Yamabe flow has been established in case of negative Yamabe invariant. Other existence results have been obtained by Roidos \cite{Conic}
in the presence of a cone singularity as long as the initial scalar curvature is in $L^{q}(M)$  for some $q>\frac{n}{2}$.
See also the work of Shao \cite{Shao}. The first two authors, jointly with Carron, extended the analysis from wedges to 
more generally, stratified spaces with iterated cone-edge singularities in 
\cite{LV, CLV}, addressing convergence of the Yamabe flow in the positive case. 
We do not attempt at providing a complete list of relevant references.

\subsection{Generalized Yamabe flow and statement of the main results}
In this work, we propose a modification of the Yamabe-flow \eqref{eq:YF} in the setting of smooth, compact manifolds. Let $(M,g_0)$ be a compact Riemannian manifold of dimension $n\geq 3$. Let $f\colon \R\to \R$ be strictly decreasing ($f'<0$) and at least $C^2$ (we will formulate more precise conditions below). Let $u_0\in C^\infty(M)$, $u_0>0$ be some fixed conformal factor. 
\begin{Def}
The \emph{generalized (normalized) Yamabe flow} is a family 
\begin{align*} g&= g(t)=u(t)^{\frac{4}{n-2}}g_0,\quad    t\in [0,T)
\end{align*}
of Riemannian metrics in the conformal class of $g_0$ such 
that the following evolution equation holds
 \begin{equation}\label{eq:GYF}
 \begin{cases} \partial_t g=(f(S)-A)g, \quad A \coloneqq \Vol_g^{-1} \int_M f(S)\, d\Vol_g,\\ 
g(0)=u_0^{\frac{4}{n-2}}g_0.\end{cases}
 \end{equation}
 \end{Def}
\noindent The family of functions $u(t)\in C^\infty(M)$ is of course assumed to be positive. 
 
The normalization by $A$ ensures that the total volume does not change along the flow. The classical Yamabe flow \eqref{eq:YF} is the special case with $f(x)=-x$.  Our new family of flows now includes flows with e.g.
$f(x) = 1/x$. \medskip

One could insist on $u_0=1$, i.e. $g(0)=g_0$, but the flexibility of potentially changing the background metric $g_0$ within its conformal class will be useful for our purposes. We will need this when discussing the short-time existence, Section \ref{Section:ShortTime},  and to have a non-trivial flow when $S_0=0$, Section \ref{Section:Flat}.  
  Since we will mainly work with $u(t)$ and not $g(t)$, it should not cause too much confusion that $g(0)\neq g_0$ in general.

In this paper we first establish short time existence of \eqref{eq:GYF} and then have different results based on the sign of the scalar curvature of the background metric, $S_0=S(g_0)$. We have simplified our discussion by only considering three cases; $S_0>0$, $S_0=0$, and $S_0<0$. We will refer to these as the \emph{scalar positive}, \emph{flat}, and \emph{negative} cases respectively. A more exhaustive treatment would be to take $Y(M,[g_0])>0$, $Y(M,[g_0])=0$, and $Y(M,[g_0])<0$ as the trichotomy instead. This would have been at the price of expositional clarity. By the solution of the strong Yamabe conjecture, one can always conformally change the metric to have constant scalar curvature $\sgn(Y(M,[g_0])$, so one could read our results as applying near this solution. Let $\tilde{S}_0$ denote the scalar curvature of $u_0^{\frac{4}{n-2}}g_0$.  We write
\begin{align*}
S_{\max}&\coloneqq \max_{p\in M} S(p),\quad S_{\min}\coloneqq \min_{p\in M} S(p)\\S_{0,\max}&\coloneqq \max_{p\in M} S_0(p),\quad S_{0,\min}\coloneqq \min_{p\in M} S_0(p),\\
\tilde{S}_{0,\max}&\coloneqq \max_{p\in M} \tilde{S}_0(p),\quad \tilde{S}_{0,\min}\coloneqq \min_{p\in M} \tilde{S}_0(p),
\end{align*}
Note that $S_{0,\max/\min}$, $\tilde{S}_{0,\max/\min}$ are constants, whereas $S_{\max/\min}$ are functions of $t$. Furthermore, $\tilde{S}_0=S_{\vert t=0}$, meaning $\tilde{S}_0$ is the initial scalar curvature. The difference between $S_0$ and $\tilde{S}_0$ disappears when $u_0=1$.

\noindent Our first main result deals with short time existence.

\begin{Thm}[Short time existence]
Let $u_0\in C^\infty(M)$, $u_0>0$. Assume $f$ is strictly decreasing on $[\tilde{S}_{0,\min},\tilde{S}_{0,\max}]$. Then there is a time $T>0$ and a unique $u(t)\in C^\infty((0,T)\times M)$, $u(t)>0$ solving  \eqref{eq:GYF}. The existence time $T$ can be extended as long as $f$ stays strictly decreasing on $[S_{\min},S_{\max}]$ and as long as $u$ stays bounded in $C^{2,\alpha}(M)$.
\end{Thm}

There are two reasons why one needs $f$ to be strictly decreasing. The first is to ensure the evolution equation is parabolic. Indeed, from the linearization one sees that \eqref{eq:GYF} is strictly parabolic if and only if $f$ is strictly decreasing. The second reason is that $f$ should at least be injective. This will ensure that a stationary solution, $0=\partial_t g= (f(S)-A)g$ implies $S=f^{-1}(A)$ is constant. \medskip

Our next main results deal with long time existence. 

\begin{Thm}[Scalar negative case]
Assume $S_0<0$ and that $f$ is strictly decreasing on  $[S_{0,\min},S_{0,\max}]$. Then there is a unique solution $u\in C^\infty((0,\infty)\times M)$ to \eqref{eq:GYF} with $u_0=1$. As $t\to \infty$, $g(t)=u(t)^{\frac{4}{n-2}} g_0$ converges to a metric with constant negative scalar curvature.
\end{Thm}

\begin{Thm}[Scalar flat and positive]
Assume that either $S_0=0$ and $u_0\in C^\infty(M)$ is positive, or that  that $S_0>0$ everywhere and $u_0=1$. Assume $f$ is bounded from below and that $f$ is strictly decreasing on $\left[\min\{\tilde{S}_{0,\min},0\}, \infty\right)$.  Then there there is a unique, positive solution $u\in C^\infty((0,\infty)\times M)$ to \eqref{eq:GYF}.
\end{Thm}

We also offer a partial result towards the long time existence in the scalar flat and positive cases under milder assumptions than boundedness of $f$. See also the discussion in Section \ref{Section:Positive} for some thoughts about pushing this further to also give an upper bound on $S$ along the flow.

\begin{Prop}[Flat case]
Assume that $S_0=0$ and that $f$ is strictly decreasing. Let $u\in C^\infty((0,T)\times M)$ be a solution to the generalized Yamabe flow with $u_0\in C^\infty(M)$ positive. Then 
\[S_{\min}\geq \tilde{S}_{0,\min}\]
and there is a constant $c>0$ which does not depend on $T$ such that 
\[c^{-1}\leq u\leq c.\]
\end{Prop}

\begin{Prop}[Positive case]
Assume $S_0\geq 0$ and that $f$ is strictly decreasing on $[0,\infty)$. Assume there are constants $a,b\geq 0$ such that $f(x)\geq -ax^{\frac{n}{2}} -b$ for $x\geq 0$. Let $u\in C^\infty((0,T)\times M)$ be the solution to \eqref{eq:GYF} with $u_0=1$ for some $T<\infty$. Then there is a constant $C_T>0$ such that 
\[
C_T^{-1}\leq u\leq C_T, \quad S_{\min}\geq S_{0,\min}e^{-C_T}.
\]
\end{Prop}

When $f$ is bounded, the proofs simplify quite a bit compared to the classical Yamabe flow $f(S)=-S$. For instance, when $f$ is bounded, \eqref{eq:YF} gives and immediate bound on $g(t)$ from below and above for all finite times. The upper bound on the scalar curvature (for finite time) when the initial scalar curvature is positive also simplifies as follows. For the classical flow, one cannot directly apply the maximum principle (there is an unfavourable sign) to deduce an upper bound on $S$. Instead, Schwetlick and Struwe \cite[pp. 68-71]{SS} get a bound on $\norm{S}_{L^p(M,g)}$ and $\norm{f(S)-A}_{L^p(M,g)}$ (when $f(S)=-S$) by carefully analysing the evolution equations of $S$. This is built upon by Brendle, \cite[Lemma 2.2, 2.3, 2.5]{Brendle}.  When $f$ is bounded, the maximum principle applied to $\log(S)$ \textit{does} give an upper bound, see Lemma \ref{Lem:SUpperBund} below.

 The structure of the paper is as follows. We start by going over some generalities of conformal flows in section \ref{Section:Flows}. Here we also deduce several evolution equations. In section \ref{sec:NonNormal} we show the short time existence of the flow by studying its linearization. The sections \ref{Section:Negative}, \ref{Section:Flat} and \ref{Section:Positive} are dedicated to the scalar negative, flat and positive case respectively. \medskip

\noindent \emph{Acknowledgements.} The second author thanks Eric Bahuaud for interesting discussions 
on the idea of generalized Yamabe flows. The third author thanks the University of Oldenburg for hospitality and 
CAPES/8881.700909/2022-01 for funding the visit to Oldenburg, during which this work was commenced. We would like to thank the referee carefully reading the paper and making several good suggestions for improvements, and for pointing out misprints.

\section{Fundamental properties of generalized Yamabe flows}
\label{Section:Flows}
Our setting is still a compact Riemannian manifold $(M,g_0)$ of dimension $n\geq 3$.  We will for convenience assume $\Vol_{g_0}=1$.  Let $f$ be some $C^2$-function and let $S$ denote the scalar curvature of $g$. 
We consider the generalised Yamabe flow as defined in \eqref{eq:GYF}, namely 
\[
\partial_t g= (f(S)-A)g, \quad g(0)=u_0^{\frac{4}{n-2}}g_0, \quad A\coloneqq\Vol_g^{-1} \int_M f(S)\, d\Vol_g.
\]
Without further assumptions on $f$, we neither know the existence nor uniqueness of solutions to \eqref{eq:GYF}. But assuming for the moment that such a flow exists, one can deduce several useful evolution equations of other geometric quantities like the volume form and the scalar curvature: \medskip

\begin{enumerate}
\item In terms of the conformal factor $u$, the flow equation \eqref{eq:GYF} is
\begin{equation}
\partial_t u=\frac{n-2}{4} (f(S)-A)\cdot u.
\label{eq:uEq}
\end{equation}

\item The volume element evolves according to
\begin{equation}
\partial_t d\Vol_{g}=\frac{n}{2} (f(S)-A)\, d\Vol_{g}
\label{eq:VolumElemEvol}
\end{equation}

\item Since the metrics $g$ and $g_0$ are related via 
\begin{equation}
g(t)=u(t)^{\frac{4}{n-2}}g_0,
\label{eq:ConfChange}
\end{equation}
 one can write down a neat formula for how the scalar curvatures relate, namely
\begin{equation}
S=u^{-\frac{n+2}{n-2}} \left(S_0\cdot u-\frac{4(n-1)}{n-2}\Delta_0 u\right)\eqqcolon u^{-\frac{n+2}{n-2}} L(u).
\label{eq:ConfScalarCurv}
\end{equation}
Here the (negative) Laplacian $\Delta_0$ is with respect to the initial metric $g_0$ and $L$ is called the conformal Laplacian.
\end{enumerate} \ \medskip

\begin{Lem}
\label{Lem:ScalarEvol}
The scalar curvature evolves along \eqref{eq:GYF} according to
\begin{equation}
\partial_t S=-(n-1) u^{-\frac{4}{n-2}}\left(\Delta_0 f(S) +2\frac{\ip{\nabla f(S)}{\nabla u}}{u}\right) -S(f(S)-A).
\label{eq:ScalarCurvEvol}
\end{equation}
The inner product is with respect to the initial metric $g_0$. 
Writing $\Delta$ for the Laplacian of $g$, this can also be written
\begin{align}
\partial_t S&=-(n-1)\Delta f(S) -Sf(S)+AS\label{eq:IntrinsicScalarCurvEvol1} \\&=-(n-1)\left(f'(S)\Delta S+f''(S) \vert \nabla S\vert^2_{g}\right) -Sf(S)+AS.
\label{eq:IntrinsicScalarCurvEvol2}
\end{align}
\end{Lem}

\begin{proof}
Both $\Delta_0$ and $S_0$ in \eqref{eq:ConfScalarCurv} are $t$-independent, so one deduces
\begin{align}
\partial_tS&\stackrel{\eqref{eq:ConfScalarCurv}}{=} -\frac{n+2}{n-2} \frac{\partial_t u}{u} S+ u^{-\frac{n+2}{n-2}} \left(S_0 \partial_t u -\frac{4(n-1)}{n-2} \Delta_{0} \partial_t u\right) \notag \\
&\stackrel{\eqref{eq:uEq}}{=}  -\frac{n+2}{4}  (f(S)-A)S+ u^{-\frac{n+2}{n-2}} \left(\frac{n-2}{4} S_0 u(f(S)-A)  -(n-1) \Delta_0 ((f(S)-A) u)\right). \notag
\end{align} 
We use \eqref{eq:ConfScalarCurv} to write $S_0u=S u^{\frac{n+2}{n-2}}+\frac{4(n-1)}{n-2} \Delta_0 u$ and insert this back into the above expression.
\begin{align}
\partial_tS&=-\frac{n+2}{4}  (f(S)-A)S+\frac{n-2}{4} (f(S)-A)S\notag 
\\&-(n-1)u^{-\frac{n+2}{n-2}}\left( \Delta_0((f(S)-A)u)-(f(S)-A) \Delta_0 u\right)\notag\\
&=-(f(S)-A)S -(n-1)u^{-\frac{n+2}{n-2}}\left( u\Delta_0 f(S)+2\ip{\nabla f(S)}{\nabla u}\right),
\end{align}
where we have also observed in the last step that $\nabla f_0=\nabla f$.
The rewriting makes use of the formula 
\begin{equation}\label{LaplacianRelation}
\Delta \xi=u^{-\frac{4}{n-2}}\left(\Delta_0 \xi +\frac{2}{u}\ip{\nabla u}{\nabla \xi}\right),
\end{equation}
which valid for any smooth function $\xi$.
\end{proof}

\noindent We note some immediate consequences via the maximum principle.
\begin{Cor}\label{scalar-min-max}
Assume  $f'(x)\leq 0$. Let $S_{\min}(t)$ and $S_{\max}(t)$ denote minima and maxima of $S$ at a given time $t$.
Then the following bounds hold.

\begin{itemize}
\item If $S_{\max}\leq 0$, then $\partial_t S_{\max}\leq 0$. \\[-3mm]

\item If  $S_{\min}\leq 0$, then $\partial_t S_{\min}\geq 0$.\\[-3mm]

\item If $\tilde{S}_{0,\min}\geq 0$, then $S_{\min}\geq 0$ (but not necessarily $\partial_t S_{\min}\geq 0$). 
\end{itemize}
\end{Cor}

\begin{proof}
The evolution equation \eqref{eq:IntrinsicScalarCurvEvol1} asserts
\[
\partial_t S=-(n-1)\Delta f(S)-S (f(S)-A).
\]
Since $f$ is decreasing, $S_{\min}$ is a maximum for $f(S)$. Similarly, 
$S_{\max}$ is a minimum for $f(S)$. Hence $\Delta f(S_{\min})\leq 0$ and 
$\Delta f(S_{\max})\geq 0$ and we conclude
\begin{equation}\label{max-ineq}
\begin{split}
&\partial_t S_{\min}\geq -S_{\min} (f(S_{\min})-A)), \\
&\partial_t S_{\max}\leq -S_{\max} (f(S_{\max})-A)).
\end{split}
\end{equation}
Since $f$ is decreasing we get $f(S_{\max})\leq f(S)\leq f(S_{\min})$, hence also $f(S_{\max})\leq A\leq f(S_{\min})$. 
This shows that 
\begin{equation}
f(S_{\min})-A\geq 0, \quad f(S_{\max})-A\leq 0.
\end{equation}
This shows the first two statements by studying the signs of 
the terms on the right hand side of \eqref{max-ineq}. Note that $S_{\min}\geq 0$
does not imply $\partial_t S_{\min}\geq 0$. Hence the third statement follows
by a different argument. Namely, we let 
\[\omega(t)\coloneqq S_{\min}(t)\exp\left(\int_0^t (f(S_{\min})-A)\, d\tau\right).\]
The differential inequality \eqref{max-ineq} for $\partial_t S_{\min}$ then says
$\partial_t \omega\geq 0$, hence
\[
S_{\min}(t)\exp\left(\int_0^t (f(S_{\min})-A)\, d\tau\right)=\omega(t)\geq \omega(0) = \tilde{S}_{0,\min}.
\]
This says that $S_{\min}$ stays non-negative as long as the flow exists.
\end{proof}

\noindent We obtain a couple of immediate consequences.

\begin{Cor}
Assume the metric evolves according to \eqref{eq:uEq} with $f'(x)\leq 0$.
\begin{enumerate}
\item Assume that the initial scalar curvature $\tilde{S}_0$ is everywhere non-positive. 
Then the scalar curvature $S$ stays non-positive along the flow.
\item Assume that the initial scalar curvature $\tilde{S}_0$ is everywhere non-negative. 
Then the scalar curvature $S$ stays non-negative along the flow.
\end{enumerate} 
\end{Cor}

If one further assumes $f$ to be bounded, one also gets an upper bound on $S$ even if $\tilde{S}_{0,\max}>0$.
\begin{Lem}
\label{Lem:SUpperBund}
Assume the metric evolves according to \eqref{eq:uEq} with $f'(x)\leq 0$ and that $f$ is bounded. Assume $\tilde{S}_{0,\max}>0$. Then there is a constant $C>0$ depending on $f$ such that
\[S_{\max}\leq \tilde{S}_{0,\max}e^{Ct}.\]
\end{Lem}
\begin{proof}
Shifting $f$ by a constant does not change the flow, so we may and do assume $f\geq 0$. 
From \eqref{max-ineq}, 
\[\partial_t \log S_{\max}\leq -f(S_{\max})+A\leq A\leq C,\]
where $C$ is an upper bound for $f$. Integrating this gives the claim. 
\end{proof}

\section{Short-time existence of generalized Yamabe flows}
\label{Section:ShortTime}

In this section, we show that the generalized Yamabe flow \eqref{eq:GYF} exists, at least for sufficiently short time. We offer two arguments. The first is a reference to "standard parabolic theory". The other, 
discussed in the appendix \ref{appendix-short-time},  is an idea for a more complete proof, relying only on linear parabolic theory, which may be the key to study generalized Yamabe flows in non-compact settings. 

\subsection{Short-time existence using standard parabolic theory}

The most general statement we could find which applies in our setting is the following.

\begin{Thm}[{\cite[Proposition 15.8.1]{Taylor}}]
\label{taylor-theorem}
Assume the problem
\[\partial_t u=F(t,x,D^2_x u),\quad u(0)=u_0\]
is strongly parabolic with $u_0\in H^s(M)$, $s>\frac{n}{2}+2$. Then there is a unique solution 
\[u\in C([0,T),H^s(M))\cap C^\infty((0,T)\times M),\]
which persists as long as $\norm{u(t)}_{C^{2+r}(M)}$ is bounded, given $r>0$. 
\end{Thm}

This will be applicable to the generalized Yamabe problem, since it is strongly parabolic as long as $f'(S)\leq -c<0$ holds. We will show this below when we compute the linearization. Since we are assuming $f$ to be strictly decreasing, we get such a bound on any compact interval.



\subsection{Non-Normalized Flow}
\label{sec:NonNormal}
Our goal is to ensure short-time existence for the normalized flow \eqref{eq:uEq}. For this, it turns out to be easier to work without the non-local term $A$. We can recover the normalised flow from the non-normalised flow by an extra argument which we sketch at the end.  \medskip

The argument for moving between the normalized and non-normalized flows becomes easier if one assumes $f$ is $\alpha-$homogeneous, by which we mean that there is $\alpha\in \R$ such that $f(\lambda x)-f(\lambda y)=\lambda^{\alpha}(f(x)-f(y))$ for all $\lambda>0$ and $x,y$ (of course, $x,y,\lambda x,\lambda y$ have to be in the domain of $f$).  In this case, we can use a trick by Hamilton,
 and we want to explain this next.  \medskip

Suppose that $v$ is a solution of the flow 
\[\begin{cases} \partial_{t}v=\frac{n-2}{4}f(R)\cdot v \\
v(0)=u_0,\end{cases}\] where $R$ denotes the scalar curvature of $\tilde{g}=v^{\frac{4}{n-2}}g_0$. We refer to the latter flow as the \textit{non-normalized flow}. The idea is to rescale both the solution $v$ and the time $t$.  Define $\eta$ to be
\begin{displaymath}
    \eta(t)= \int_{0}^{t}\frac{1}{\Vol_g(M)}\int_{M}f(R)dVol_{g}d\tau.
    \label{eq:etaDef}
\end{displaymath}
Let $u=\exp\left(-\frac{n-2}{4}\eta\right)v$. Let $S$ denote the scalar curvature of $g=u^{\frac{4}{n-2}}g_0$. Since 
\[g=u^{\frac{4}{n-2}} g_0 = e^{-\eta} v^{\frac{4}{n-2}} g_0=e^{-\eta} \tilde{g}\] and $\eta$ only depends on time, we have $S=e^{\eta} R$. Then we compute 
\begin{displaymath}
\begin{aligned}
    \partial_{t}u&= \left(\partial_{t}v-\frac{n-2}{4}\eta'v\right)\exp\left(-\frac{n-2}{4}\eta\right)\\
    &=\frac{n-2}{4}\left(f(R)v-\eta'v\right)\exp\left(-\frac{n-2}{4}\eta\right)\\
    &=\frac{n-2}{4}\left(f(R)-\eta'\right)\exp\left(-\frac{n-2}{4}\eta\right)v\\
    &=\frac{n-2}{4}\left(f(Se^{-\eta})-\frac{1}{\Vol_g(M)}\int_{M}f(Se^{-\eta})dVol_{g}\right)u.   
\end{aligned}
\end{displaymath}
We now use the assumption that $f$ is $\alpha-$homogeneous (which we do not generally assume - it is just to demonstrate Hamilton's trick).
  Multiply both sides by $e^{\alpha\eta}$ and find
\[e^{\alpha\eta}\partial_t u=\frac{n-2}{4}(f(S)-A)u.\]
Rescaling the time by letting $\tau$ be the solution of $\frac{d\tau}{dt}=e^{-\alpha \eta(t)}$, $\tau(0)=0$,
 we get a solution to the normalized flow
 \[\partial_{\tau} u=\frac{n-2}{4}(f(S)-A)u.\] 
 
\subsection{Linearization of the flow}
\label{Sec:Linearisation}
Let us consider the non-normalized flow
\begin{equation}\label{eq:NonuEq}
\partial_{t}u=f(S)\cdot u, \quad u(0)=u_0.
\end{equation}
Let $\beta\coloneqq \frac{n+2}{n-2}$ and $c_{n}\coloneqq \frac{4(n-1)}{n-2}$. In \eqref{eq:ConfScalarCurv} we saw  
\begin{equation}\label{eq:Constants}
\begin{aligned}
    S=&u^{-\beta}L(u),\\
    L=&S_{0}-c_{n}\Delta_{0}.
\end{aligned}
\end{equation}
\begin{Lem}[Frechét derivative]\label{lemmaFrechet}
Consider $\mathcal{F}(u)=f(S)\cdot u$. Then the Frechét derivative of $\mathcal{F}$ is
\begin{equation}
   D\mathcal{F}(u)= f\left(S\right)+f'\left(S\right)\left(u^{1-\beta}L-\beta S\right). 
   \label{eq:Frechet}
\end{equation}
\end{Lem}
\begin{proof}
In order to linearize \eqref{eq:NonuEq} we replace $u$ by $u+h$ and expand everything in powers of $h$. For instance 
\begin{equation}\label{eq:Binomial}
    (u+h)^{-\beta}=u^{-\beta}-\beta u^{-(\beta+1)}h+\mathcal{O}(h^{2}).
\end{equation}
Also, from \eqref{eq:Binomial},  we compute
\begin{displaymath}
\begin{aligned}
    (u+h)^{-\beta}L(u+h)
    =u^{-\beta}L(u)+u^{-(\beta+1)}(uL(h)-\beta hL(u))+\mathcal{O}(h^{2}).
\end{aligned}
\end{displaymath}
Computing the Taylor expansion of $f$ around $u^{-\beta}L(u)$ gives us
\begin{displaymath}
\begin{aligned}
f\left((u+h)^{-\beta}L(u+h)\right)(u+h)
&=f\left(u^{-\beta}L(u)\right)(u+h)+\mathcal{O}(h^{2})\\
&+f'\left(u^{-\beta}L(u)\right)\left(u^{-(\beta+1)}(uL(h)-\beta hL(u))\right)u.
\end{aligned}
\end{displaymath}
From \eqref{eq:Constants} we conclude
\begin{displaymath}
\begin{aligned}
& f\left((u+h)^{-\beta}L(u+h)\right)(u+h)-f\left(u^{-\beta}L(u)\right)u\\
 &=f\left(S\right)h+f'\left(S\right)\left(u^{1-\beta}L(h)-\beta Sh\right)
 +\mathcal{O}(h^{2}),   
\end{aligned}
\end{displaymath}
and since $\mathcal{F}(u)=f(S)\cdot u$, the Fréchet derivative is
\begin{displaymath}
   D\mathcal{F}(u)h= f\left(S\right)h+f'\left(S\right)\left(u^{1-\beta}L(h)-\beta Sh\right).
\end{displaymath}
\end{proof}

\begin{Cor}
The flow \eqref{eq:NonuEq} is strictly parabolic as long as $f'(S)\leq -c<0$.
\label{Cor:Parabolic}
\end{Cor}
\begin{proof}
A non-linear PDE is called strictly parabolic if its linearization is strictly parabolic. From \eqref{eq:Frechet}, the linearization has $-c_nf'(S)u^{1-\beta} \Delta_0$ as highest order differentiation.
\end{proof}

\subsection{The normalised flow}
At the start of this section, we explained how to get a normalized solution if the function $f$ is $\alpha-$homogeneous. We do not generally assume this, so we give an argument to incorporate the normalization.  The most direct strategy would be to redo the linearization argument with the additional term $-Au$ dragged along. The linearization can be deduced as follows. Write 
\[A(u)\coloneqq \frac{1}{\Vol_g} \int_M f(S)\, d\Vol_g,
\quad D\Vol_g (h)\coloneqq \frac{2n}{n-2}\int_M h u^{-1} d\Vol_g.\]
Let $\mathcal{N}(u)=A(u)u.$ Then
\begin{align*}
D\mathcal{N}(h)&= \frac{1}{\Vol_g}\left( -D\Vol_g(h)A(u)+\int_M f'(S)(u^{1-\beta}L(h)-\beta S h) 
\frac{}{} \right.  \\ &\left.  \frac{}{} +\frac{2n}{n-2}f(S) h u^{-1}\, d\Vol_g\right)u +A(u)h.
\end{align*}
This has the structure of
\[h\mapsto \int_M \mathcal{D}(h)\, d\Vol_g+A(u)h,\]
where $\mathcal{D}$ is an elliptic operator. The normalized flow therefore linearizes to
\[\partial_t h=\mathcal{P} h +\int_M \mathcal{D}(h)\, d\Vol_g\]
for the elliptic operator $\mathcal{P}=D\mathcal{F}(u)+A(u)$.

\section{Long-time existence and convergence: the negative case}
\label{Section:Negative}

The goal of this section is to show the long-time existence and convergence of the generalized Yamabe flow when the initial scalar curvature $S_0$ is negative.  We need the framework of Hölder spaces which we now recall.

\begin{Def}[Spaces]\label{def-spaces}
Let $0<\alpha< 1$ and $(M,g_{0})$ a compact, smooth manifold of dimension $n\geq 3$. We denote the Riemannian distance function by $d\colon M\times M\to \R$. As usual, we introduce the parabolic distance $d_P\colon ([0,T]\times M)^2 \to \R$ as 
\[d_P((t,p),(s,q))\coloneqq \sqrt{d(p,q)^2+\vert t-s\vert}.\]
We define the $\alpha$-th Hölder norm by 
\[\norm{u}_{C^{\alpha}([0,T]\times M)}\coloneqq \sup_{[0,T]\times M} \vert u\vert + \sup_{(t,p)\neq (s,q)} \frac{\vert u(t,p)-u(s,q)\vert}{d_P((t,p),(s,q))^{\alpha}}\]
The Hölder spaces $C^\alpha([0,T]\times M)$ and $C^{2,\alpha}\left([0,T]\times M\right)$ are then
\begin{displaymath}
\begin{aligned}
&C^\alpha([0,T]\times M)\coloneqq \left\{u\in C([0,T]\times M)\, \colon\, \norm{u}_{C^{\alpha}([0,T]\times M)}<\infty\right\}, \\
&C^{2,\alpha}\left([0,T]\times M\right)\coloneqq \left\{w\in C^{\alpha}\left([0,T]\times M\right)\, \colon\, \Delta_{0}w\in C^{\alpha}\left([0,T]\times M\right)\right\},
\end{aligned}
\end{displaymath}
where the latter space is equipped with the norm\footnote{This is equivalent to the usual definition of $C^{2,\alpha}$ by elliptic regularity, e.g. \cite[Theorem 8.22]{GT}.}
\[\norm{u}_{C^{2,\alpha}([0,T]\times M)}\coloneqq \norm{u}_{C^\alpha([0,T]\times M)}+\norm{\Delta_0 u}_{C^\alpha([0,T]\times M)}.\] 
The spaces $C^\alpha(M)$ and $C^{2,\alpha}(M)$ are defined in the same way with $T=0$.
We shall also need the parabolic Hölder space
\[C^{2+\alpha}([0,T]\times M)\coloneqq \left\{u\in C^{2,\alpha}([0,T]\times M)\, \colon\, \partial_t u \in C^{\alpha}([0,T]\times M)\right\}\]
with the norm
\[\norm{u}_{C^{2+\alpha}}\coloneqq \norm{u}_{C^{2,\alpha}([0,T]\times M)}+\norm{\partial_t u}_{C^\alpha([0,T]\times M)}.\]

Finally, we need the open set
\begin{equation}
\label{eq:BoundedAwayZero2}
\mathcal{O}\coloneqq \left\{w\in C^{2,\alpha}\left(M\right)\, \colon\, \exists\, c,\tilde{c}>0: w>c \ \ \mbox{and} \ \ f'(w^{-\beta}Lw)<-\tilde{c}\right\}.
\end{equation}
\end{Def}

\subsection{Uniform estimates of solutions}

Throughout this section, we keep using the notation: 
\begin{displaymath}
\begin{aligned}
S_{0,\min}\coloneqq \min_{p\in M}S_{0}(p)\\
S_{0,\max}\coloneqq \max_{p\in M}S_{0}(p). 
\end{aligned}
\end{displaymath}
We will also set 
\[u_0=1\]
in this section. Hence $S_0$ is the initial scalar curvature, $\tilde{S}_0=S_0$.
\begin{Prop}\label{Prop:EstimateGeneralizedScalar}
Assume $S_{0}$ is everywhere negative and let $f\in C^{2}(\mathbb{R})$ be  strictly decreasing on $[S_{0,\min},S_{0,\max}]$. Then there exist positive constants $B,C$ such that any solution of the generalized Yamabe flow in $C^{2+\alpha}([0,T)\times M)$ satisfies
\begin{displaymath}
\norm{f(S)-A}_{L^\infty(M)}\leq Ce^{-Bt}.
\end{displaymath}
The constants $B,C$ are independent of $T$.
\end{Prop}
\begin{proof} 
From Corollary \ref{scalar-min-max}, we know that $S\in [S_{0,\min},S_{0,\max}]$ for any time.  Applying the maximum and minimum principle to \eqref{eq:IntrinsicScalarCurvEvol2} tells us
\begin{displaymath}
\begin{aligned}
\partial_{t}S_{\max}&\leq -S_{\max}\left(f(S_{\max})-A)\right),\\
\partial_{t}S_{\min}&\geq-S_{\min}\left(f(S_{\min})-A)\right).
\end{aligned}
\end{displaymath}
As in the proof of Corollary \ref{scalar-min-max}, $f(S_{\max})-A\leq0$ and $f(S_{\min})-A\geq0$, thus
\begin{displaymath}
\begin{aligned}
\partial_{t}(S_{\min}-S_{\max})&\geq S_{\max}\left(f(S_{\max})-A)\right)-S_{\min}\left(f(S_{\min})-A)\right)\\
&\geq S_{0,\max}\left(f(S_{\max})-A)\right)-S_{\min}\left(f(S_{\min})-A)\right)\\
&=S_{0,\max}(f(S_{\max})-f(S_{\min}))+(S_{0,\max}-S_{\min})(f(S_{\min})-A)\\
&\geq S_{0,\max}(f(S_{\max})-f(S_{\min})).
\end{aligned}
\end{displaymath}
By the mean value theorem there exist $\xi_{t}\in(S_{\min},S_{\max})\subset[S_{0,\min},S_{0,\max}]$ such that $\displaystyle{f'(\xi_{t})=\frac{f(S_{\max})-f(S_{\min})}{S_{\max}-S_{\min}}}$. Thus,
\begin{displaymath}
\begin{aligned}
\partial_{t}(S_{\min}-S_{\max})\geq S_{0,\max}f'(\xi_{t})(S_{\max}-S_{\min}).
\end{aligned}
\end{displaymath}
Let \[-c=\max_{\xi\in[S_{0,\min},S_{0,\max}]} f'(\xi), \quad B\coloneqq -c S_{0,\max}.\]
Note that $B$ is positive since both $S_0$ and $f'$ are negative. So
\begin{displaymath}
    \partial_{t}(\ln{(S_{\max}-S_{\min})})\leq -B.
\end{displaymath}
Integrating this inequality from $0$ to $t$ and applying the exponential we conclude
\begin{equation}\label{eq:expBoundScalarDiff}
(S_{\max}-S_{\min})(t)\leq (S_{0,\max}-S_{0,\min})e^{-Bt}=c_{0}e^{-Bt}
\end{equation}
with $c_{0}\coloneqq S_{0,\max}-S_{0,\min}$. Since $f$ is decreasing,
\begin{displaymath}
\begin{aligned}
f(S_{\max})\leq &f(S)\leq f(S_{\min}),\\
-f(S_{\min})\leq &-A\leq -f(S_{\max}).
\end{aligned}
\end{displaymath}
So
\begin{displaymath}
-(f(S_{\min})-f(S_{\max}))\leq f(S)-A\leq f(S_{\min})-f(S_{\max}).
\end{displaymath}
From \eqref{eq:expBoundScalarDiff}, and the mean value theorem again, we conclude
\begin{displaymath}
\vert f(S)-A\vert\leq \vert f(S_{\max})-f(S_{\min}\vert=\vert f'(\xi_{t})\vert(S_{\max}-S_{\min})\leq c_{0}\vert f'(\xi_{t})\vert e^{-Bt}.
\end{displaymath}
Take $\displaystyle{C=\max_{\xi\in[S_{0,\min},S_{0,\max}]} \vert f'(\xi)\vert \cdot c_
{0}}$ and conclude 
\begin{displaymath}
\norm{f(S)-A}_{L^\infty(M)}\leq Ce^{-Bt}.
\end{displaymath}
\end{proof}
\begin{Prop}\label{BoundSolution} Assume $S_{0}$ is everywhere negative and let $f\in C^{2}(\mathbb{R})$ be strictly decreasing on $[S_{0,\min},S_{0,\max}]$.
Suppose that $u$ is a solution of the generalized Yamabe flow in $C^{2+\alpha}([0,T)\times M)$ with $u_0=1$. Then there exists a constant $c>0$, depending on $A(0)$,  $S_{0,\max}$,  $S_{0,\min}$ but not on $T$, such that $c^{-1}\leq u(p,t)\leq c$ for all $p\in M$ and $t\in[0,T)$. 
\end{Prop}
\begin{proof}
Since $u_0=1$ is positive, we know that $u$ remains positive for all $t$ small enough.
From the generalized Yamabe flow, we find
\begin{displaymath}
    \partial_{t}\ln{u}=\frac{n-2}{4}(f(S)-A).
\end{displaymath}
Hence, from Proposition \ref{Prop:EstimateGeneralizedScalar},
\begin{displaymath}
|\partial_{t}\ln{u}|\leq\frac{n-2}{4}Ce^{-Bt}
\end{displaymath}
i.e,
\begin{displaymath}
-\frac{n-2}{4}Ce^{-Bt}\leq\partial_{t}\ln{u}\leq\frac{n-2}{4}Ce^{-Bt}.
\end{displaymath}
Integrating over $[0,t]$ gives us
\begin{displaymath}
-\frac{(n-2)C}{4B}\left(1-e^{-Bt}\right)\leq\ln{u(t)}\leq\frac{(n-2)C}{4B}\left(1-e^{-Bt}\right).
\end{displaymath}
Therefore,
\begin{displaymath}
    \left|\ln{u}\right|\leq\frac{(n-2)C}{4B},
\end{displaymath}
i.e.
\begin{equation}\label{eq:BoundfunctionU}
    e^{-\frac{(n-2)C}{4B}}\leq u(t)\leq e^{\frac{(n-2)C}{4B}}.
\end{equation}
If $u$ were to not remain positive during the flow, there would $\tilde{T}<T$ such that $u(t)>0$ for all $t\in[0,\tilde{T})$ and $u_{\min}(\tilde{T})=0$. Proceeding as before, we obtain the inequality \eqref{eq:BoundfunctionU} for $u(\tilde{T})$. Hence $u_{\min}(\tilde{T})>0$, a contradiction. So $u$ remains positive during the flow.
\end{proof}

\begin{Cor}
Assume the setup of Proposition \ref{BoundSolution}. Then there exists a constant $\tilde{c}>0$, depending on  $A(0)$,   $S_{0,\min}$, $S_{0,\max}$ but independent of $T$, such that $\norm{\partial_{t}u}_{L^{\infty}(M)}\leq \tilde{c}e^{-Bt}$. 
\end{Cor}
\begin{proof}
The generalized Yamabe flow equation is 
\begin{displaymath}
\partial_{t}u=\frac{n-2}{4}(f(S)-A)\cdot u,
\end{displaymath}
and we have derived uniform bounds on both factors on the right hand side.
\end{proof}

\subsection{Long-time existence}
We are ready to complete the long-time existence argument.
\begin{Thm}[Long-time existence]\label{Thm:NegLongTime}
Assume the setup of Proposition \ref{BoundSolution}. Then the generalized Yamabe flow exists for all time.
\end{Thm}
\begin{proof}
The idea is to show that if the maximal existence time $T$ were finite, the limit $\lim_{t\to T} u(t)$ exists and is in $\mathcal{O}$ (defined in \eqref{eq:BoundedAwayZero2}). So one can restart the flow, contradicting the maximality of $T$. We already have uniform $L^{\infty}$ bounds on $u$ and $u^{-1}$ by Proposition \ref{BoundSolution}. We also know by Corollary \ref{scalar-min-max} that $S\in [S_{0,\min},S_{0,\max}]$, so $f'(S)\leq -c$ uniformly in time. What remains to show are uniform $C^{2,\alpha}(M)$-bounds on $u$. We will show these in three steps.
\bigskip

 \noindent\textbf{Step 1 -- $u\in C^{1,\alpha}(M)$ uniformly:}
\medskip

\noindent By \eqref{eq:ConfScalarCurv}, we may write
\begin{equation}
-c_n\Delta_0 u=u^\beta S-S_0 u.
\label{eq:Delta_0u}
\end{equation}
The right hand side is uniformly bounded. So 
\[u\in W^{2,p}(M)\coloneqq \{w\in L^p(M)\, \colon\, \Delta_0 w\in L^p(M)\}\]
uniformly for any $p\geq 1$. By the Sobolev embedding theorem (more precisely, Morrey's inequality), $u\in C^{1,\alpha}(M)$ uniformly for any $\alpha\in (0,1)$. \bigskip

\noindent\textbf{Step 2 -- $S\in C^\alpha(M)$ uniformly:}
\medskip

\noindent Multiply the evolution equation for $S$, \eqref{eq:IntrinsicScalarCurvEvol1}, by $f'(S)$ to deduce
\begin{equation}\label{PDE-KrylovSafonov}
(\partial_t +(n-1)f'(S)\Delta)f(S)=-Sf'(S)(f(S)-A).
\end{equation}
Since $f'(S)\leq -c$ uniformly, the right hand side is uniformly bounded. From \eqref{LaplacianRelation}
\begin{displaymath}
f'(S)\Delta \xi=u^{-\frac{4}{n-2}}f'(S)\left(\Delta_0 \xi +\frac{2}{u}\ip{\nabla u}{\nabla \xi}\right)
\end{displaymath}
thus the operator $(n-1)f'(S)\Delta$ is uniformly elliptic, since by \textbf{Step 1} $u\in C^{1,\alpha}(M)$ uniformly and  $f'(S)\leq -c$ uniformly.  So the PDE \eqref{PDE-KrylovSafonov} is  parabolic with uniformly bounded coefficients. By Krylov-Safanov estimates, \cite[Theorem 4.3]{KS}, and the compactness of $M$, we deduce that $f(S)$ is uniformly in  $C^\alpha(M)$. By the mean value theorem there exist $\xi_{t}\in(S(t,p),S(t,q))\subset[S_{0,\min},S_{0,\max}]$ such that $\displaystyle{f'(\xi_{t})(S(t,p)-S(t,q))=f(S(t,p))-f(S(t,q))}$. Thus,
\begin{displaymath}
|(S(t,p)-S(t,q))|\leq \frac{1}{c}|f(S(t,p))-f(S(t,q))|,    
\end{displaymath}
and
\begin{displaymath}
\begin{aligned}
||S(t)||_{C^{\alpha}(M)}&=||S(t)||_{L^{\infty}(M)}+\sup_{p\neq q}\frac{|S(t,p)-S(t,q)|}{d(p,q)^{\alpha}}\\
&\leq||S(t)||_{L^{\infty}(M)}+\frac{1}{c}\sup_{p\neq q}\frac{|f(S(t,p))-f(S(t,q))|}{d(p,q)^{\alpha}}<\infty,
\end{aligned}
\end{displaymath}
so $S$ is also uniformly $C^\alpha$.   \bigskip

\noindent\textbf{Step 3 -- $\Delta_0 u\in C^\alpha(M)$ uniformly:}
\medskip

\noindent Returning to \eqref{eq:Delta_0u}, we conclude that $\Delta_0 u$ is uniformly $C^\alpha$.\bigskip

Let $0<\alpha'<\alpha$. Then $C^\alpha(M)\subset C^{\alpha'}(M)$ compactly, so there is a subsequence $t_k\to T$ such that $u(t_k)$ converges. Call this limit $u(T)$. By the uniform bounds, $u$ is continuos up to $T$, hence does not depend on the chosen subsequence.  The limit is in $\mathcal{O}$ by the above considerations. 
\end{proof}

\subsection{Convergence of the generalized Yamabe flow}

\begin{Thm}[The scalar negative case]
Assume that $g_{0}$ is a metric with scalar curvature $S_{0}$ everywhere negative and suppose that $u(t)$ is the solution to generalized Yamabe flow \eqref{eq:uEq} (with $u_0=1$) that exists for all time. Then the associated metric $g(t)=u(t)^{\frac{4}{n-2}}g_{0}$ converges to a metric with constant negative scalar curvature.
\end{Thm}
\begin{proof}
The associated metric $g(t)=u(t)^{\frac{4}{n-2}}g_{0}$ solves the flow
\begin{displaymath}
    \partial_{t}g=(f(S)-A)g.
\end{displaymath}
Since the flow exists for all time, Proposition \ref{Prop:EstimateGeneralizedScalar} says
\begin{displaymath}
    f(S)-A\rightarrow0
\end{displaymath}
at an exponential rate. Since $u(t)$ is uniformly bounded, also $\partial_{t}g\rightarrow0$ exponentially. Thus $g$ converges to a continuous limit metric $g^{*}=(u^{*})^{\frac{4}{n-2}}g_{0}$, and the conformal factor $u(t)$ admits a continuous pointwise limit $u^{*}$ as $t\rightarrow\infty$. Proceeding as in the proof of Theorem \ref{Thm:NegLongTime} we conclude $u^{*}\in C^{2,\alpha}(M)$ and then the limit metric $g^{*}$ admits a well defined scalar curvature. Since $f(S)-A$ vanishes in the limit $t\rightarrow\infty$, we conclude that the scalar curvature of the limit metric must be constant, since $f$ is invertible with a $C^1$ inverse on the interval $[S_{0,\min},S_{0,\max}]$. The limiting sectional curvature is negative by Corollary \ref{scalar-min-max}.
\end{proof}

\section{Long-time existence: the flat case}
\label{Section:Flat}

The goal of this section is to show the long-time existence of the generalized Yamabe flow when the initial metric $g_0$ is scalar flat. To get a non-trivial flow, we therefore need to allow some non-trivial conformal factor, $u_0\neq 1$. \medskip

\noindent \textbf{Important notation:}
Recall the notation $\tilde{S}_0$ for the scalar curvature of $u_0^{\frac{4}{n-2}} g_0$, which of course does not have to vanish. We write 
\[\tilde{S}_{0,\min}\coloneqq \min\limits_{p\in M} \tilde{S}_0(p), \quad \tilde{S}_{0,\max}\coloneqq\max\limits_{p\in M} \tilde{S}_0(p).\]
We will keep using the abbreviations $\beta=\frac{n+2}{n-2}$ and $c_n=\frac{4(n-1)}{n-2}$.

\subsection{Uniform estimates of solutions}

 \begin{Lem}\label{LowerBoundScalar}
Assume $(M,g_0)$ is scalar flat and suppose that $u$ is a solution of the generalized Yamabe flow in $C^{2+\alpha}([0,T)\times M)$. Then $0 \geq  S_{\min}\geq \tilde{S}_{0,\min}$.
\end{Lem}
\begin{proof}
We first argue that $S$ is either $0$ or has both signs. The scalar curvature is given by
\[S=u^{-\beta} L(u)=-c_n u^{-\beta} \Delta_0 u,\]
since $g_0$ is scalar flat. Multiplying both sides by $u^{\beta}$ and integrating tells us
\[\int_M u^\beta S\, d\Vol_{g_0}=-c_n \int_M \Delta_0 u\, d\Vol_{g_0}=0.\]
Since $u>0$, either $S=0$ everywhere or $S$ is both positive and negative. $S=0$ is trivial, so we assume $S_{\min}<0$. 
The claim then follows from Corollary \ref{scalar-min-max}.
\end{proof}

\begin{Prop}\label{BoundSolutionFlat} Assume $(M,g_0)$ is scalar flat and let $f\in C^{2}(\mathbb{R})$ be strictly decreasing.
Suppose that $u$ is a solution of the generalized Yamabe flow in $C^{2+\alpha}([0,T)\times M)$. Then there exists a constant $c>0$, depending on $u_0$ and the scalar flat metric $g_{0}$, but not on $T$, such that $c^{-1}\leq u(p,t)\leq c$ for all $p\in M$ and $t\in[0,T)$. 
\end{Prop}
\begin{proof}
This is similar to the proof in \cite{Ye}.
As above, 
\[S=u^{-\beta} L(u)=-c_n u^{-\beta}\Delta_0 u.\] Let $u_{\max}$ and $u_{\min}$ denote the maximum and minimum of $u$ at a fixed time $t$.
Then
 \begin{displaymath}
u_{\max}^{-\beta}L(u_{\max})=-u^{-\beta}_{\max}c_{n}\Delta_{0}u_{\max}\geq 0
\end{displaymath}
and
\begin{displaymath}
u_{\min}^{-\beta}L(u_{\min})=-u^{-\beta}_{\min}c_{n}\Delta_{0}u_{\min}\leq 0.
\end{displaymath}
Since $f$ is strictly decreasing,
 \begin{displaymath}
 f(u_{\max}^{-\beta}L(u_{\max}))-A\leq  f(u_{\min}^{-\beta}L(u_{\min}))-A
 \end{displaymath}
 Set $N=\frac{4}{n-2}$. Inserting these bounds into the generalized Yamabe flow then tells us
\begin{displaymath}
\begin{aligned}
\frac{1}{N}\partial_{t}\ln{u_{\max}}
=f(u_{\max}^{-\beta}L(u_{\max}))-A\leq f(u_{\min}^{-\beta}L(u_{\min}))-A=\frac{1}{N}\partial_{t}\ln{u_{\min}}.
\end{aligned}
 \end{displaymath}
Integrating this over $[0,t]$ yields
\begin{displaymath}
\frac{u_{\min}(0)}{u_{\max}(0)}u_{\max}\leq u_{\min}.
\end{displaymath}
So we have proven that there exists a constant $0<k<1$ such that the Harnack inequality
\begin{equation}\label{eq789}
    u_{\min}(t)\geq k^{\frac{n-2}{2n}} u_{\max}(t)
\end{equation}
holds. As
\begin{displaymath}
    Vol_{g}=\int_{M}u^{\frac{2n}{n-2}}dVol_{g_{0}}
\end{displaymath}
and the volume is constant along the flow, the Harnack inequality \eqref{eq789} yields,
\begin{equation}\label{eq790}
 u_{\min}^{\frac{2n}{n-2}}(t)\geq\frac{k}{Vol_{g_{0}}}\int_{M}u_{\max}^{\frac{2n}{n-2}}(t)dVol_{g_{0}}\geq\frac{k}{Vol_{g_{0}}}\int_{M}u^{\frac{2n}{n-2}}(t)dVol_{g_{0}}=k.
\end{equation}
Thus $u_{\min}$ is uniformly bounded away from zero. Similarly, $u_{\max}^{\frac{2n}{n-2}}(t) \leq \frac{1}{k}$,
and then $u_{\max}$ is uniformly bounded from above.\\
If $u$ were to not remain positive during the flow, there would $\overline{T}<T$ such that $u(t)>0$ for all $t\in[0,\overline{T})$ and $u_{\min}(\overline{T})=0$. Proceeding as before, we obtain the inequality \eqref{eq790}. Hence $u_{\min}(\overline{T})>0$, a contradiction. So $u$ remains positive during the flow.
\end{proof}

%
To get long-time existence, we assume that $f$ is bounded.

\begin{Thm}[Long-time existence]\label{Thm:FlatLongTime}
Assume the setup of Proposition \ref{BoundSolutionFlat}. Assume that $f$ is bounded. Then the generalized Yamabe flow exists for all time.
\end{Thm}

\begin{proof}
Since $f$ is bounded, we get an upper bound on the scalar curvature by Lemma \ref{Lem:SUpperBund}.
The proof is then essentially the same as for Theorem \ref{Thm:NegLongTime}. 
\end{proof}

\section{Long-time existence: the positive case}
\label{Section:Positive}

In this section, we will prove the long time existence of the flow when the initial scalar curvature is positive. To this end, we need bounds valid for finite time. We will set
\[u_0=1\]
again.

From Corollary \ref{scalar-min-max}, we know that $S\geq 0$. The plan is to use this to deduce an upper and a lower bound on $u$. We will show that this will follow once we have a lower bound on $A(t)$, $A(t)\geq a$ for some $a\in \R$. This is the case in the classical case, $f(x)=-x$, since then Lemma \ref{Lem:TotalScalarEvol} says
\[\sigma'(t)=-\frac{n-2}{2}\int_M (S-\sigma)^2\, d\Vol_g\leq 0.\]
So $a=-\sigma(0)$ is the lower bound. A lower bound also trivially follows if one assumes $f$ is bounded from below. More generally, we offer the following assumptions.

\begin{Assump}

The initial scalar curvature $S_0$ and the function $f$ satisfy:
\begin{itemize}
\item The initial scalar curvature $S_0$ is positive, $S_0>0$;
\item $f$ is strictly decreasing;
\item There are constants $\mu,\nu\geq 0$ and $1\leq \kappa\leq \frac{n}{2}$ such that 
\[-f(x)\leq \mu x^{\kappa} +\nu\]
for all $x\geq 0$;
\end{itemize}
\label{Assump:Positive}
\end{Assump}

\begin{Rem}
The case $\mu=0$ means that $-f$ is bounded. The arguments become quite a bit simpler in this case, and we will go over them later in the section.

One could probably drop the first assumption and instead work with $S+\xi$ for $\xi$ large enough as the Brendle does in \cite{Brendle}, \cite{BrendleYF} and the first two authors do in \cite{CLV}. 
\end{Rem}

\begin{Lem}
\label{Lem:ALower}
Impose Assumption \ref{Assump:Positive}.
 Then there is $a\in \R$ such that  $a\leq A=\int f(S)\, d\Vol_g$ independent of time.    
\end{Lem}

\begin{proof}
By Corollary \ref{scalar-min-max}, we know $S\geq 0$ for all time. Integrating the assumed bound on $f$, we find
\[-A=\int_M -f(S)\, d\Vol_g\leq \mu \int_M S^{\kappa}+\nu \, d\Vol_g=\mu \norm{S}^{\kappa}_{L^{\kappa}(M,g)}+\nu.\] 
Since $\kappa\leq \frac{n}{2}$ per assumption,  we can use Corollary \ref{Cor:Sn/2Bound} to find a time-independent bound for the right hand side. 
\end{proof}

\begin{Cor}
\label{Cor:SPos}
Assume the setup of Lemma \ref{Lem:ALower}. Assume\footnote{This can always be arranged by replacing $f\mapsto f-f(0)$. Adding a constant to $f$ does not change the normalized flow.} $f(0)=0$. Then we have
\[S_{\min} \geq S_{0,\min} e^{a t}\]
for any $t$. In particular, $S_{0}>0\implies S>0$ for any finite time.
\end{Cor}
\begin{proof}
Follows directly from the proof of Corollary \ref{scalar-min-max}. There we showed
\[S_{\min}\geq S_{0,\min} \exp\left(\int_0^t (A-f(S_{\min}))\, d\tau\right)\geq S_{0,\min}e^{a t},\]
where we have used $f(0)=0\implies f(S)\leq 0$. 
\end{proof}

\begin{Prop}
\label{Prop:PositiveuBounds}
Assume $A(t)\geq a$ for some $a\in \R$ for all $t\in [0,T]$. Then there is a constant $C(T)>0$ such that $\frac{1}{C(T)}\leq u\leq C(T)$ for $t\in [0,T]$.
\end{Prop}

\begin{proof}
 \noindent\textbf{Step 1 -- upper bound on $u$:}
\medskip

 The evolution equation for $u$, \eqref{eq:uEq}, says
\[\partial_t u  =\frac{n-2}{4}(f(S)-A) u.\]
Since $f$ is decreasing and $S\geq 0$, we can bound the right hand side by $ f(S)-A\leq f(0)-a \eqqcolon c$. Integrating, we find for $t\in [0,T]$
\[u\leq \exp\left(\frac{n-2}{4}  c T\right).\]

\noindent\textbf{Step 2 -- lower bound on $u$:}
\medskip
Recall that \eqref{eq:ConfScalarCurv} says
\[-c_n \Delta_0 u+S_0 u=u^{\beta} S,\]
where $\beta=\frac{n+2}{n-2}$ and $c_n=\frac{4(n-1)}{n-2}$. Let $\alpha>0$. Then
\[-c_n \Delta_0 u^{-\alpha}=\alpha c_n u^{-\alpha-1}\Delta_0 u-c_n \alpha (\alpha-1)u^{-\alpha-2}\vert \nabla u\vert^2_{g_0}\leq \alpha c_n u^{-\alpha-1}\Delta_0 u.\]
By the above, 
\[\alpha c_n u^{-\alpha-1}\Delta_0 u=\alpha \left(S_0 u^{-\alpha}-u^{\beta-\alpha-1}S\right)\leq \alpha S_0 u^{-\alpha}.\]
So
\[-c_n \Delta_0 u^{-\alpha} \leq \alpha S_0 u^{-\alpha}.\]
The function $v=u^{-\alpha}$ is therefore a subsolution to an elliptic equation,
\[(c_n \Delta_0 +\alpha S_0)v \geq 0.\] By the Harnack inequality, \cite[Theorem 8.17]{GT} for instance\footnote{The quoted theorem is for a ball in $\R^n$. Since $M$ is compact, we can cover it by finitely many sets diffeomorphic to balls in $\R^n$ and apply the theorem to each ball.}, we get a bound
\[u^{-\alpha}\leq C  (u^{-\alpha})_{\min}=C u_{\max}^{-\alpha}.\]
for some constant $C>0$. The right hand side is bounded since the volume is constant along the flow;
\[1=\Vol_g(M)=\int_M u^{\frac{2n}{n-2}}d\Vol_{g_0}\leq u_{\max}^{\frac{2n}{n-2}}\int d\Vol_{g_0}=u_{\max}^{\frac{2n}{n-2}}.\]
So
\[u^{-\alpha}\leq C.\]

\end{proof}

A corollary of these results is that if $S_0>0$, then $f(S)=-S^{\kappa}$ is an allowable function for any $1\leq \kappa\leq \frac{n}{2}$.  

\subsection{Upper bound on the scalar curvature when $f$ is bounded}

We now assume $f$ is bounded from below. I.e. Assumption \ref{Assump:Positive} with $\mu=0$. Changing $f$ by adding some constant $C$ does not change the flow. So assuming $f$ is bounded from below means we may (and will) assume $f(x)\geq 0$ for all $x\geq 0$. 
\begin{Rem}
 The classical Yamabe flow, $f(x)=-x$, does \textit{not} have this property. Functions which do satisfy Assumption \ref{Assump:Positive} with $\mu=0$ include $f(x)=\exp(-\alpha x)$, $f(x)=\exp\left(-(x+\alpha)^\beta\right)$, and $f(x)=(x+\alpha)^{-\beta}$, for $\alpha,\beta>0$.
\end{Rem}

\begin{Thm}
\label{Thm:PosLongTime}
Assume $S_0>0$ and let $f$ be as in Assumption \ref{Assump:Positive} with $\mu=0$. Then the generalized Yamabe flow \eqref{eq:uEq} has a unique solution $u$ for any $t\in [0,\infty)$.
\end{Thm}

The idea is of course to get uniform $C^{2,\alpha}$-bounds on the solution $u$ for any finite time. We first prove a couple of lemmas. They all use the same assumptions as Theorem \ref{Thm:PosLongTime}, but we will not keep writing this. 

\begin{Lem}
We have $S(t)\geq 0$ uniformly for any finite time interval $[0,T]$ and $S_{\max}(t)\leq C(T) S_{0,\max}$.
\label{Lem:PosScalarBound}
\end{Lem}
\begin{proof}
We did the lower bound above. The upper bound is Lemma \ref{Lem:SUpperBund}.
\end{proof}

%
%
%
\begin{Rem}
The constant $C(T)$ is exponentially growing in $T$. 
We note how much weaker these bounds are compared to the negative scalar curvature case, Proposition \ref{Prop:EstimateGeneralizedScalar}. In particular, they are of no use as $t\to \infty$.
\end{Rem}

Since $f$ is bounded from below, Assumption \ref{Assump:Positive} holds and we may use Proposition \ref{Prop:PositiveuBounds}. But one can also give a more direct argument:
\begin{Lem}
The solution $u(t)$ to the generalized Yamabe flow \eqref{eq:uEq} (with $u(0)=1$ satisfies
\[\exp\left(-\frac{4}{n-2} f(0)t\right) \leq u(t)\leq \exp\left(\frac{4}{n-2} f(0)t\right).\]
\end{Lem}

\begin{proof}
By Assumption \ref{Assump:Positive} and Lemma \ref{Lem:PosScalarBound}, $f(S)\geq 0$ for all time. So 
\[0\leq f(S)\leq f(0), \quad 0\leq A\leq f(0)\] 
for all time. Hence
\[-\frac{4}{n-2}f(0)u\leq \frac{4}{n-2} (f(S)-A)u=\partial_t u =\frac{4}{n-2} (f(S)-A)u\leq \frac{4}{n-2} f(0)u.\]
Integrating this gives the claim.  
\end{proof}

\begin{proof}[Proof of Theorem \ref{Thm:PosLongTime}]
We can use the same proof as for Theorem \ref{Thm:NegLongTime} with the following modifications. 
For any finite time, $T$, Lemma \ref{Lem:PosScalarBound} says that $S([0,T])$ lands inside a compact subset of $(0,\infty)$ (but not necessarily $[S_{0,\min},S_{0,\max}]$). By Assumption \ref{Assump:Positive}, there is a $c=c(T)>0$ such that $f'(S)\leq -c$ for $t\in [0,T]$. The rest of the proof goes through verbatim. 
\end{proof}

\subsection{Some ideas when not assuming $f$ to be bounded}
We end the section with some ideas of what to try if one does not assume $f$ to be bounded. If one keeps Assumption \ref{Assump:Positive}, one only has to get an upper bound on the scalar curvature. We will additionally assume the Yamabe constant to be positive\footnote{This is not a big assumption. If $Y(M,[g_0])\leq 0$, one can appeal to \cite{Trudinger} to conformally change the metric to have negative or vanishing scalar curvature, which would land us in one of the easier cases.}
\[Y(M,[g_0])>0.\]

 We suggest the following steps. Prove:
\begin{enumerate}
\item  There is $q>\frac{n}{2}$ such that $S\in L^q(M,g)$ for all time.
\item  There is the following Sobolev inequality. There are time-independent constants $A,B>0$ such that 
\begin{equation}
\label{eq:StrongYamabeSobolev}
A\norm{\phi}^2_{L^{\frac{2n}{n-2}}(M,g)}\leq \norm{\nabla \phi}^2_{L^2(M,g)}+B\norm{\phi}^2_{L^2(M,g)}
\end{equation}
holds for any $\phi\in H^1(M,g)$.
\item For any $T<\infty$, there is a constant $C(T)$ such that
\[\int_0^T \left( \int_M S^{\frac{n^2}{2(n-2)}}\, d\Vol_g\right)^{\frac{n-2}{n}}\, dt\leq C(T).\]
\item For any $T<\infty$, there is a constant $C(T)$ such that
\[S\leq C(T).\]
\end{enumerate}

We do not attempt to prove the steps (1) and (4) here.  We suspect one could get them from the evolution equations of Appendix \ref{Appendix} as in the classical case (see \cite[pp. 68-71]{SS},  \cite[Lemma 2.3, Proposition 2.6]{Brendle}). This is somewhat intricate as the evolution equations are messier for the generalized Yamabe flow. Another idea to get (4) from (1)-(3) is to use Moser iteration similarly to \cite[Theorem 4.1]{LV}.

\begin{Lem}[Step (2)]
Let $q>\frac{n}{2}$. Assume there is a time-independent constant $C>0$ such that $\norm{S}_{L^q(M,g)}\leq C$. Then there are time-independent constants $A,B>0$ such that 
\begin{equation}
\notag
A\norm{\phi}^2_{L^{\frac{2n}{n-2}}(M,g)}\leq \norm{\nabla \phi}^2_{L^2(M,g)}+B\norm{\phi}^2_{L^2(M,g)}
\end{equation}
holds for any $\phi\in H^1(M,g)$.
\end{Lem}
\begin{proof}
The conformal invariance of the Yamabe constant $Y=Y(M,[g_0])$ gives us a similar bound:
\begin{equation}
Y\norm{\phi}_{L^{\frac{2n}{n-2}}(M,g)}^2\leq c_n \norm{\nabla \phi}^2_{L^2(M,g)}+\int_M S \phi^2\, d\Vol_g.
\label{eq:YamabeSobolev}
\end{equation}

We estimate the last term using the H\"{o}lder inequality with $q$ and $p=\frac{q}{q-1}$;
\[\int_M S \phi^2\, d\Vol_g\leq \norm{S}_{L^q(M,g)} \norm{\phi ^2}_{L^p(M,g)}\leq C\norm{\phi^2}_{L^p(M,g)},\]
where we have inserted the assumed bound on $ \norm{S}_{L^q(M,g)}$. Since $q>\frac{n}{2}$, we have $1<p<\frac{n}{n-2}$, and we may interpolate between the $L^1$ and the $L^{\frac{n}{n-2}}$-norms as follows\footnote{The general statement is this. Fix $p_0<p<p_1$ and choose $\theta$ so that $\frac{1}{p}=\frac{1-\theta}{p_0}+\frac{\theta}{p_1}$. Then $\norm{\phi}_{L^p}\leq \norm{\phi}_{L^{p_0}}^{1-\theta}\norm{\phi}^{\theta}_{L^{p_1}}$, and one checks this by applying the H\"{o}lder inequality to $\phi=\phi^{1-\theta} \phi^{\theta}$.}  
\[\norm{\phi^2}_{L^p(M,g)}\leq \norm{\phi^2}_{L^1(M,g)}^{1-\theta} \norm{\phi^2}^\theta_{L^{\frac{n}{n-2}}(M,g)},\]
where $\theta=\frac{n}{2q}<1$. To this product we apply Young's inequality $ab\leq \theta (\epsilon^\theta a)^{\frac{1}{\theta}} +(1-\theta)(\epsilon^{-\theta}b)^{\frac{1}{1-\theta}}$ for any $\epsilon>0$ to deduce
\[\norm{\phi^2}_{L^p(M,g)}\leq  \theta \epsilon \norm{\phi}^2_{L^{\frac{2n}{n-2}}(M,g)}+ (1-\theta) \epsilon^{-\frac{\theta}{1-\theta}} \norm{\phi}^2_{L^2(M,g)}.\]
Inserting this back into \eqref{eq:YamabeSobolev} leaves us with
\[Y \norm{\phi}^2_{L^{\frac{2n}{n-2}}(M,g)}\leq c_n\norm{ \nabla \phi}^2_{L^2(M,g)} + C\left( \theta \epsilon \norm{\phi}^2_{L^{\frac{2n}{n-2}}(M,g)}+ (1-\theta) \epsilon^{-\frac{\theta}{1-\theta}} \norm{\phi}^2_{L^2(M,g)}\right),\]
which can be written as
\[\left( Y- C \theta \epsilon\right) \norm{\phi}^2_{L^{\frac{2n}{n-2}}(M,g)}\leq c_n\norm{ \nabla \phi}^2_{L^2(M,g)}+ C(1-\theta) \epsilon^{-\frac{\theta}{1-\theta}} \norm{\phi}^2_{L^2(M,g)}.\]
Choosing $\epsilon$ small enough ensures the left hand side is non-negative (here we are using $Y(M,g_0)>0$) and we deduce \eqref{eq:StrongYamabeSobolev}. 
\end{proof}
\begin{Rem}
The above corollary is a general fact about the Yamabe constant's relation to a Sobolev inequality. It has nothing to the with the particular flow under study. The flow enters the discussion when showing that the assumption $\norm{S}_{L^q(M,g)}\leq C$ is satisfied.
\end{Rem}

\begin{Lem}[Step (3)]
For any $T<\infty$, there is a constant $C(T)$ such that
\[\int_0^T \left( \int_M S^{\frac{n^2}{2(n-2)}}\, d\Vol_g\right)^{\frac{n-2}{n}}\, dt\leq C(T).\]
\end{Lem}

\begin{proof}
The evolution for $\norm{S}_{L^p(M,g)}$ of Lemma \ref{Lem:SpEvol} with $p=\frac{n}{2}$ says
\[\frac{d}{dt} \norm{S}^{\frac{n}{2}}_{\frac{n}{2}}=\frac{4(n-2)(n-1)}{n}\int_M f'(S) \vert \nabla S^{\frac{n}{4}}\vert\, d\Vol_g.\]
We know $S>0$ for any finite time due to Corollary \ref{Cor:SPos}. By assumption \ref{Assump:Positive}, we have a uniform bound $f'(S)\leq -c$ (depending on $T$). Using the Sobolev inequality \eqref{eq:StrongYamabeSobolev} tells us
\[\frac{d}{dt} \norm{S}^{\frac{n}{2}}_{\frac{n}{2}}\leq -c_1 \norm{S^{\frac{n}{4}}}^2_{\frac{2n}{n-2}}+c_2 \norm{S^\frac{n}{4}}^2_2.\]
Now 
\[\norm{S^{\frac{n}{4}}}^2_2=\norm{S}_{\frac{n}{2}}^{\frac{n}{2}}\leq \norm{S_0}_{\frac{n}{2}}^{\frac{n}{2}}\]
by Corollary \ref{Cor:Sn/2Bound}. Integrating both sides from $0$ to $T$ then tells us
\[\int_0^T  \norm{S^{\frac{n}{4}}}^2_{\frac{2n}{n-2}}\, dt\leq C_1\left(\norm{S_0}^{\frac{n}{2}}_{\frac{n}{2}}-\norm{S}^{\frac{n}{2}}_{\frac{n}{2}}\right)+C_2T\norm{S_0}^{\frac{n}{2}}_{\frac{n}{2}}\leq (C_1+C_2T)\norm{S_0}^{\frac{n}{2}}_{\frac{n}{2}}.\]
This is the claim with $C(T)=(C_1+C_2T)\norm{S_0}^{\frac{n}{2}}_{\frac{n}{2}}$.
\end{proof}

\appendix

\section{Short-time existence via linearization}
\label{appendix-short-time}

In this section we outline an alternative argument for 
obtaining short time existence of the generalized Yamabe flow 
without referring to Theorem \ref{taylor-theorem},  an approach that might prove useful 
in the non-compact setting as in e.g. \cite{ShortTime}.
In addition to the spaces introduced in Section \ref{Section:Negative},  we need one more open set
\begin{align}\label{eq:BoundedAwayZero1}
\mathcal{O}_T&\coloneqq \left\{w\in C^{2+\alpha}\left([0,T]\times M\right)\, \colon\, \exists\, c,\tilde{c}>0: w>c \ \ \mbox{and} \ \ f'(w^{-\beta}Lw)<-\tilde{c}\right\}.
\end{align}

\begin{Prop}[Linearized problem]\label{linearizedequation}
We start by discussing the non-normalized flow. 
Let $0<\alpha< 1$ and assume $f\in C^{2}(\mathbb{R})$ is strictly decreasing. Then for any $u\in\mathcal{O}_T$, $h_0\in C^{2,\alpha}(M)$ and $\Psi\in C^{\alpha}([0,T]\times M)$ there exists a unique $h\in C^{2+\alpha}([0,T]\times M)$ which is a solution to
\begin{equation}\label{eq:linearizedequation}
\begin{aligned}
\partial_{t}h&=D\mathcal{F}(u)h+\Psi\\
h(0)&=h_0
\end{aligned}
\end{equation}
where $\mathcal{F}$ is defined in Lemma \ref{lemmaFrechet}.
\end{Prop}
\begin{proof}
Note that $\mathcal{A}\coloneqq D\mathcal{F}(u)= f\left(S\right)+f'\left(S\right)\left(u^{1-\beta}L-\beta S\right)$ is a strongly elliptic operator for each $u\in\mathcal{O}_T$. Its local coefficients are $C^\alpha$ functions.  The solution will therefore follow by standard parabolic theory as follows. Cover $M$ by finitely many coordinate neighbourhoods $U_i$, which we identify with opens (also called $U_i$) in $\R^n$. Let $\phi_i$ be a partition of unity subordinate to $\{U_i\}$. Let $h_{0,i}\coloneqq \phi_i h_0$ and $\Psi_i\coloneqq \phi_i \Psi$. By the local theory for linear parabolic PDEs, \cite[Theorem 7, Chapter 3]{Friedman} for instance, there are unique solutions $h_i$ to the initial-boundary value problems
\[\begin{cases} (\partial_t -\mathcal{A})h_i=\Psi_{i} \\ (h_{i})_{t=0}=h_{i,0} \\ (h_i)_{\vert \partial U_i}=0.\end{cases}\]
These solutions are in $C^{2+\alpha}([0,T]\times U_i)$. Extend these solutions to continuous functions on $M$ by extending by $0$;
\[\overline{h}_i(x,t)\coloneqq \begin{cases} h_i(x,t) & x\in U_i \\ 0 & x\in M\setminus U_i.\end{cases}\]
Let  $h\coloneqq \sum_i \overline{h}_i$. We claim $h$ is the sought solution. It clearly satisfies the right initial condition since $\sum_{i} \phi_i=1$.  To see that it solves the PDE, let $v\in C^{2+\alpha}([0,T]\times M)$ and let $\mathcal{A}^*$ denote the adjoint of $\mathcal{A}$. Then
\begin{align*}
\ip{h}{\mathcal{A}^*v}_{L^2(M)}&=\sum_i \ip{\overline{h}_i}{\mathcal{A}^* v}_{L^2(M)}=\sum_i \ip{h_i}{\mathcal{A}^* v}_{L^2(U_i)}\\
&=\sum_i\ip{\mathcal{A} h_i}{v}_{L^2(U_i)}=\sum_i\ip{\partial_t h_i-\Psi_{i}}{v}_{L^2(U_i)}\\
&=\ip{\partial_t h-\Psi}{v}_{L^2(M)}.
\end{align*}
This shows  that $h$ is a weak solution. Hence it is also the strong solution. 
\end{proof}
\begin{Thm}[Short-time existence]
Let $0<\alpha< 1$ and assume $f\in C^{2}(\mathbb{R})$ is strictly decreasing. Let $u_0\in \mathcal{O}$. 
Then there exists a $T>0$ and a unique $u\in C^{2+\alpha}([0,T]\times M)$ solving
\begin{equation}\label{eq:PVI}
\begin{aligned}
\partial_{t}u&=f(S)\cdot u\\
u(0)&=u_{0}.
\end{aligned}
\end{equation}
\end{Thm}
\begin{proof}
Let $T>0$.
We introduce the map
\begin{align*}
\tF\colon \mathcal{O}_T&\to C^1([0,T];C^\alpha( M))\\
w &\mapsto w(t)-\int_0^t \mathcal{F}(w)\, d\tau. 
\end{align*}
The reason being that a solution $u$ is precisely a function satisfying $\tF(u)=u(0)$ for any $t$. So we can produce the solution $u$ if we can invert $\tF$. The linearization $D\tF$ at a point $w\in \mathcal{O}_T$ is readily seen to be
\[D\tF(w)h(t)=h(t)-\int_0^t D\mathcal{F}(w)h\, d\tau.\]
Since we can solve the linearised problem at $w=u_0$ (we think of $u_0$ as an element of $\mathcal{O}_T$ which is constant in time) by Proposition \ref{linearizedequation}, $D\tF$ is locally invertible.  
From the definition,
\[\tF(u_0)=(1-tf(u_0^{-\beta}L(u_0)))u_0\]
is in $C^1([0,T];C^{\alpha}( M))$ and is close to $u_0$ for $t$ small enough. By the inverse function theorem $\tF$ is therefore invertible near $u_0$ for $t\in [0,T]$ with $T$ small enough.  
\end{proof}
For the normalized flow in the case where $f$ is not $\alpha$-homogeneous, we computed the linearisation towards the end of Section \ref{Sec:Linearisation}. 
One would need to quote a slightly stronger result than \cite{Friedman} to allow for the integral term, but other than that, one could repeat the argument of the non-normalized flow. 

\begin{Rem}
Another strategy one could pursue would be to set up a fixed point argument. Here is a rough suggestion
using Hölder spaces,  introduced in Definition \ref{def-spaces} below.
 Let $K=\left\{v\in C^{2,\alpha}([0,T]\times M)\, \colon\, \exists\, c>0\, \colon\, c\leq v\leq c^{-1}\right\}\cap \overline{B_1}(u_0)$, where 
\[B_1(u_0)=\{v\in C^{2,\alpha}([0,T]\times M)\, \colon \norm{v-u_0}_{C^{2,\alpha}}\leq 1\}.\]
This is clearly a closed and convex subset. For $v\in K$, set $d\Vol_{g_v}\coloneqq v^{\frac{2n}{n-2}} d\Vol_{g_0}$ and $S_{v}\coloneqq v^{-\beta} L(v)$. This the volume form and scalar curvature respectively of the metric $g_v\coloneqq v^{\frac{4}{n-2}}g_0$.  Define
\begin{align*}
A\colon K \to C^\alpha([0,T]), \quad 
A(v)\coloneqq \frac{1}{\Vol_{g_v}(M)} \int_M f(S_v)\, d\Vol_{g_v}.
\end{align*}
Finally, let $\mathcal{G}\colon K\to C^{2+\alpha}([0,T]\times M)$ be defined as 
\[\mathcal{G}(v)\coloneqq \text{solution of } \begin{cases} \partial_t u=(f(S)-A(v))u \\ u(0)=u_0.\end{cases}\]
Here $S=S_u$ is the scalar curvature of $g=u^{\frac{4}{n-2}}g_0$. This map is well-defined by the short-time existence of the non-normalised flow with a harmless additional linear term.
Furthermore, $\mathcal{G}$ is continuous. A solution to the normalised flow is precisely a fixed point of $\mathcal{G}$. One would then have to show the existence of a fixed point, for instance by showing that $\mathcal{G}$ is a contraction.
\end{Rem}

\section{Some more evolution equations}
\label{Appendix}

We now proceed with some a priori $L^p$-estimates in the positive case.
We assume as before that $\Vol_{g_0}(M)=1$ to simplify some formulas. 
Recall 
\[A=\int_M f(S)\, d\Vol_g.\] We also introduce the average scalar curvature
$$
\sigma\coloneqq \int_M S\, d\Vol_g.
$$

\begin{Lem}
\label{Lem:TotalScalarEvol}
Along the generalized Yamabe flow \eqref{eq:uEq}, we have
\begin{equation}
\begin{split}
A'(t) &= (n-1)\int_M f'(S)f''(S)\vert \nabla S\vert^2\, d\Vol_g + \int_M \left(\frac{n}{2}f(S)-Sf'(S)\right)(f(S)-A)\, d\Vol_g, \\
\sigma'(t) &= \frac{n-2}{2}\int_M S(f(S)-A)\, d\Vol_g=\frac{n-2}{2}\int_M (S-\sigma)(f(S)-A)\, d\Vol_g.
\end{split}
\end{equation}
\end{Lem}

\begin{proof}
We have
\[A'(t)=\int_M f'(S) \partial_t S\, d\Vol_g + \int_M f(S) \partial_t d\Vol_g.\]
The first term we rewrite with \eqref{eq:IntrinsicScalarCurvEvol2} and an integration by parts,
\begin{align*}
\int_M f'(S) \partial_t S\, d\Vol_g &=\int_M f'(S)\left(-(n-1)\left(f'(S)\Delta S+f''(S) \vert 
\nabla S\vert^2_{g}\right) -S(f(S)-A)\right)\, d\Vol_g \\
&=(n-1)\int_M f'(S)f''(S)\vert \nabla S\vert^2_g \, d\Vol_g -\int_M f'(S)S(f(S)-A)\, d\Vol_g.
\end{align*}
The second term is \eqref{eq:VolumElemEvol}. Combining these gives the claimed formula.
The evolution equation for $\sigma$ is a bit simpler, namely
\[
\sigma'(t)=\int_M \partial_t S\, d\Vol_g + \int_M S \, \partial_t d\Vol_g,
\]
and inserting \eqref{eq:IntrinsicScalarCurvEvol1} and \eqref{eq:VolumElemEvol} gives the claim.
\end{proof}

\begin{Lem}
\label{Lem:SpEvol}
Let $p\geq 2$. Then the $L^p$-norms of $f(S), S$,  $(S-\sigma)$, and $(f(S)-A)$ evolve according to 
\begin{align*}
\frac{d}{dt} \int_M \vert f(S)\vert^p\, d\Vol_g &= (n-1)\int_M f'(S)\vert f(S)\vert^{p-2} \left((p-1)f'(S)^2+f(S)f''(S)\right)\vert \nabla S\vert^2\, d\Vol_g\\
&+\left(\frac{n}{2}f(S)^2 -pSf(S)f'(S)\right)\vert f(S)\vert^{p-2} (f(S)-A)\,d\Vol_g, \\
\frac{d}{dt} \int_M \vert S\vert^p \, d\Vol_g &=p(p-1)(n-1)\int_M \vert S\vert ^{p-2} f'(S) \vert \nabla S\vert^2\, d\Vol_g \\
&+\left(\frac{n}{2}-p\right)\int_M \vert S\vert^p (f(S)-A)\, d\Vol_g, \\
\frac{d}{dt} \int_M \vert S-\sigma\vert^p\, d\Vol_g&=p(p-1)(n-1)\int_M f'(S)\vert S-\sigma\vert^{p-2}\vert \nabla S\vert^2\, d\Vol_g \\
&+\left(\frac{n}{2}-p\right)\int_M (f(S)-A)\vert S-\sigma\vert^{p}\, d\Vol_g\\
&-p\int_M \left(\sigma'(t)+\sigma(f(S)-A)\right)\vert S-\sigma\vert^{p-2}(S-\sigma)\, d\Vol_g.\\
\frac{d}{dt} \int_M \vert f(S)-A\vert^p \, d\Vol_g&=p(n-1)\int_M \vert f(S)-A\vert^{p-2}(f(S)-A)f'(S)f''(S)\vert \nabla S\vert^2_g \, d\Vol_g\\
&+p(p-1)(n-1)\int_M \vert f(S)-A\vert^{p-2}(f'(S)^3\vert \nabla S\vert^2_g \, d\Vol_g \\
 &+\int_M \vert f(S)-A\vert^p\left(\frac{n}{2}(f(S)-A)-pS f'(S)\right)\, d\Vol_g \\
 &-p A'(t) \int_M \vert f-A\vert^{p-2}(f(S)-A)\, d\Vol_g.
\end{align*}

\end{Lem}

\begin{proof}
The identities follow by similar arguments as in Lemma 
\ref{Lem:TotalScalarEvol}. The only addition is the formula 
$\nabla \vert S\vert^p =p \vert S\vert^{p-2} S \nabla S$.
\end{proof}

\begin{Cor}
\label{Cor:Sn/2Bound}
Abbreviating $L^p=L^p(M,g)$, we have
\[\norm{S}_{L^p} \leq \norm{S_0}_{L^{\frac{n}{2}}}\]
for any $p\leq \frac{n}{2}$.
\end{Cor}
\begin{proof}
The evolution equation in Lemma \ref{Lem:SpEvol} for $\norm{S}_{L^p}$ says 
\[
\frac{d}{dt} \int_M \vert S\vert^{\frac{n}{2}}\, d\Vol_g =
 \frac{n(n-1)(n-2)}{4}\int_M  \vert S\vert^{\frac{n}{2}-2}f'(S) \vert \nabla S\vert^2 d\Vol_g\leq 0.
 \]
This proves the claim for $p=\frac{n}{2}$. The case $p < \frac{n}{2}$ follows from the H\"{o}lder inequality
\[
\norm{S}_{L^p}^p=\norm{S^p}_{L^1}\leq \norm{S^p}_{L^{\frac{n}{2p}}}\norm{1}_{L^{\frac{n}{n-2p}}}=\norm{S}^p_{L^{\frac{n}{2}}},
\]
  where we used the volume normalisation $\Vol_g=1$.
\end{proof}

\section*{Statements and Declarations}
The authors declare that they have no conflicts of interest to disclose.


\end{document}